\documentclass[11pt,reqno]{amsart}

\usepackage{geometry}

\title{Eichler-Shimura Relation on Intersection Cohomology}
\author{Zhiyou Wu}

\usepackage{amsmath,amssymb,fancyhdr,url,amsthm,mathtools,mathrsfs,eufrak,setspace}
\usepackage[pdftex]{graphicx}
\usepackage{tikz-cd}
\usetikzlibrary{cd}
\usepackage[mathscr]{euscript}
\usepackage{tikz}
\usepackage{dsfont}
\usetikzlibrary{matrix}

\makeatletter
\newcommand*{\rom}[1]{\expandafter\@slowromancap\romannumeral #1@}
\makeatother

\usepackage{relsize} 
\usepackage[bbgreekl]{mathbbol} 
\usepackage{amsfonts} 
\DeclareSymbolFontAlphabet{\mathbb}{AMSb} 
\DeclareSymbolFontAlphabet{\mathbbl}{bbold}

\usepackage[toc,page]{appendix}



\newtheorem{proposition}{Proposition}[section]
\newtheorem{theorem}[proposition]{Theorem}

\newtheorem{corollary}[proposition]{Corollary}

\newtheorem{remark}[proposition]{Remark}
\newtheorem{lemma}[proposition]{Lemma}

\begin{document}

\maketitle

\begin{abstract}
We prove the Eichler-Shimura relation on intersection cohomolgoy of minimal compactifications of Shimura Varieties of Hodge type. 
\end{abstract}

\section{Introduction}

Let $G$ be a reductive group over $\mathbb{Q}$, and $\mu: \text{Res}_{\mathbb{C}/\mathbb{R}} \mathbb{G}_m \longrightarrow G_{\mathbb{R}}$ a cocharacter giving rise to Shimura varieties $X_K$ parameterized by level groups $K \subset G(\mathbb{A}_f)$. Let $p$ be a prime over which $G$ has a reductive model $G_{\mathbb{Z}_p}$ over $\mathbb{Z}_p$, so $X_{K^p}:= X_{K^pG_{\mathbb{Z}_p}(\mathbb{Z}_p)}$ has good reduction for $K^p\subset G(\mathbb{A}_f^p)$ small enough.   

Let $l$ be a prime different from $p$. 
In
\cite{MR1265563}, Blasius and Rogawski constructed an explicit polynomial $H_{G,\mu}(T)$  with coefficients in Hecke algebra 
$\overline{\mathbb{Q}_l}[G_{\mathbb{Z}_p}(\mathbb{Z}_p)\setminus G(\mathbb{Q}_p)/G_{\mathbb{Z}_p}(\mathbb{Z}_p)]$, and
conjectured that the Frobenius action $\Psi$ on 
\[
IH^*(X_{K^p}^{\text{min}}, \overline{\mathbb{Q}_l})
\]
(with its natural Hecke algebra module structure)
satisfying 
\begin{equation} \label{pquieopiwup}
H_{G,\mu}(\Psi) =0,
\end{equation}
where $X_{K^p}^{\text{min}}$ is the minimal compactification of $X_{K^p}$ and $IH$ is the intersection cohomology. This is a generalization of the classical Eichler-Shimura relation $T_p = F+V$ for modular curves. 

This conjecture has been established in many cases, but with intersection cohomology replaced by the usual (or compactly supported) cohomology of the noncompactified Shimura varieties,  see
\cite{MR1770603}, \cite{MR3270785},  \cite{2020arXiv200611745L}, and \cite{2021arXiv211010350W} for example. Most approaches to the Blasius-Rogawski conjecture follow the strategy of Faltings-Chai (\cite{FaltingsChai} chap. \rom{7}) by establishing the equation (\ref{pquieopiwup}) as algebraic correspondences. One exception is \cite{2021arXiv211010350W}, in which (\ref{pquieopiwup}) is established as cohomological correspondences by interpreting Hecke operators as excursion operators.

In \cite{2020arXiv200611745L}, it is explained how to deduce  (\ref{pquieopiwup}) for intersection cohomology from (\ref{pquieopiwup}) as algebraic correspondences. This cannot be directly applied to the setting of \cite{2021arXiv211010350W}. In this paper, we supply an alternative approach to establish  (\ref{pquieopiwup})
for intersection cohomology so we can deduce from \cite{2021arXiv211010350W} the Eichler-Shimura relation for intersection cohomology of (minimal compactifications of) Shimura varieties of Hodge type.

More precisely, we deduce the Blasius-Rogawski conjecture from a formal result on extending cohomological correspondences to the compactifications. 

\begin{theorem} \label{eouril}
Let $X$ be a smooth scheme over finite field $\mathbb{F}_q$, and $\overline{X}$ a proper scheme over $\mathbb{F}_q$ which contains $X$ as an open dense subscheme. Let $\mathbb{L}$ be a pure local system on $X$, and  $P(T)$ be a polynomial with coefficients being cohomological correspondences from $\mathbb{L}$ to itself. Moreover, we assume that the support of the coefficients are proper over $X$ under the right leg (the one with !-pullback).

Suppose that $\Psi$ is a cohomological correspondence from $\mathbb{L}$ to itself whose support is proper over $X$ under the right leg, and such that 
\[
P(\Psi) = 0,
\]
we refer to section \ref{fjdlseiq} for the potential subtle addition of cohomological correspondences. 
Then the coefficients of $P(T)$ and $\Psi$ naturally acts on $H^*(\overline{X}, \text{IC}(\mathbb{L}))$, and the relation 
\[
P(\Psi) = 0
\]
holds in the endormorphism ring of $H^*(\overline{X}, \text{IC}(\mathbb{L}))$. 
\end{theorem}

\begin{remark}
The restriction on the base being finite fields comes from the use of Morel's weight t-structure in the proof. It is possible to generalize the result to the base of number field or complex numbers by applying the analogous formalism of weight t-structures in \cite{morel2019mixed} and  \cite{nairmixed}. 
\end{remark}

As the Eichler-Shimura relation in \cite{2021arXiv211010350W} is established as cohomological correspondences, we have the desired corollary. 

\begin{corollary} \label{ES}
The Eichler-Shimura relation holds on the intersection cohomology of Shimura varieties of Hodge type. 
\end{corollary}

\begin{remark}
The Eichler-Shimura relation in \cite{2021arXiv211010350W} is established with constant coefficients. If we can prove the Eichler-Shimura relation with automorphic coefficient systems on open Shimura varieties (as cohomological correspondences), then we can also deduce the relation for intersection cohomology with automorphic coefficient systems. This has been proved in certain special cases in \cite{2020arXiv200611745L}, and our result applies to their contexts. 
\end{remark}

\subsection*{Acknowledgement}
I would like to thank Cheng Shu and Yifeng Liu for refreshing my interests on the questions considered in this paper.

\section{Reivew of Cohomological Correspondences} 

Let $X$, $Y$ and $Z$ be separated schemes of finite type over a field $k$, together with a diagram 
\[
\begin{tikzcd}
& Z \arrow[rd,"p_2"] \arrow[ld,"p_1"] &
\\
X & & Y
\end{tikzcd}
\]
Let $\mathcal{F} \in D(X, \overline{\mathbb{Q}_l})$ and $\mathcal{G} \in D(Y,\overline{\mathbb{Q}_l})$
where $l$ is different from the characteristic of $k$. Then a cohomological correspondence from $\mathcal{F}$ to $\mathcal{G}$ supported on $Z$ is a 
morphism
\[
c: p_1^* \mathcal{F} \longrightarrow p_2^! \mathcal{G}
\]
in $D(Z, \overline{\mathbb{Q}_l})$. We can  compose cohomological correspondences, provided the domain and codamain being compatible.

\subsection{Pushforward of cohomological correspondences}

Now given a commutative diagram 
\[
\begin{tikzcd}
& Z_1 \arrow[rd,"p_2"] \arrow[ld,"p_1"] \arrow[dd,"f"]
&
\\
X_1 \arrow[dd,"h"]& & Y_1 \arrow[dd,"g"]
\\
& Z_2 \arrow[rd,"q_2"] \arrow[ld,"q_1"] &
\\
X_2 & & Y_2
\end{tikzcd}
\]
and let $c: p_1^* \mathcal{F} \longrightarrow p_2^! \mathcal{G}$ be a cohomological correspondence supported on $Z_1$. 

Assume that

$\bullet$ \textit{$p_2$ and $q_2$ are proper,} \\
then we can pushforward $c$ along the diagram to obtain the cohomological correspondence
\[
f_* c: q_1^*h_*\mathcal{F} \longrightarrow  q_2^! g_* \mathcal{G}
\]
supported on $Z_2$. It is defined as the composition
\[
q_1^*h_*\mathcal{F} \longrightarrow f_*p_1^* \mathcal{F} \overset{f_* c}{\longrightarrow }
f_*p_2^!\mathcal{G} \longrightarrow 
q_2^! g_* \mathcal{G},
\]
where the first map is the base change map, and the last map is the adjoint of the map 
\[
q_{2!}f_*p_2^!\mathcal{G} = q_{2*}f_*p_2^!\mathcal{G} =g_*p_{2*}p_2^!\mathcal{G} = g_*p_{2!}p_2^!\mathcal{G}
\longrightarrow g_* \mathcal{G}.
\]
The identities above makes use of that $p_2$ and $q_2$ are proper, and the last arrow is the counit map. 

It is easy to check that $(f \circ f')_* c = f_*(f'_*c)$, provided that the correspondences satisfy our hypothesis on properness of the right leg. Moreover, we can also see directly that pushforward commutes with composition of cohomological correspondences, again under the properness hypothesis. 

\begin{remark} \label{fjlsqw}
There are two other cases where one can pushforward cohomological correspondences, namely, 

1) $f$ and $g$ are proper, 

2) The right square is Cartesian. \\
It is easy to see that pushforward in this two cases still preserves composition. Moreover, it is functorial. 
\end{remark}

\subsection{Extension of cohomological correspondences} \label{928309isj}
Let 
\[
c: p_1^* \mathcal{F} \longrightarrow p_2^! \mathcal{G}
\]
be a cohomological correspondence supported on
\[
\begin{tikzcd}
& Z \arrow[rd,"p_2"] \arrow[ld,"p_1"] &
\\
X & & X.
\end{tikzcd}
\]
 Let $\overline{X}$  be a proper scheme over  $k$  such that $X \subset \overline{X}$ is an open dense subscheme. Then we can choose a compactification $\overline{Z}$ of $Z$, i.e. $Z \subset \overline{Z}$ open dense and $\overline{Z}$ is proper, together with two morphism $q_1, q_2: \overline{Z} \rightarrow \overline{X}$ fitting into a commutative diagram
\[
\begin{tikzcd}
& Z \arrow[rd,"p_2"] \arrow[ld,"p_1"] \arrow[dd,hook,"j_{Z}"]
&
\\
X \arrow[dd,hook,"j_X"]& & X \arrow[dd,hook,"j_X"]
\\
& \overline{Z} \arrow[rd,"q_2"] \arrow[ld,"q_1"] &
\\
\overline{X} & & \overline{X}
\end{tikzcd}
\]
For example, we can choose an arbitrary compactification $Z^*$ of $Z$ (exists by Nagata), and take $\overline{Z}$ to be the closure of the image of $Z$ in $Z^* \times \overline{X} \times \overline{Y}$. 
\begin{lemma} \label{Fujiwa}
(\cite{1997InMat.127..489F} lemma 1.3.1) Assume that $p_2$ is proper (so that $j_{Z*} c$ is well-defined), then $j_{Z*} c$ is  the unique cohomological correspondence  from $j_{X*}\mathcal{F}$ to $j_{X*} \mathcal{G}$ supported on $\overline{Z}$ whose restriction on $Z$ is $c$.  
\end{lemma}

We can also see that the extension is independent of the choice of $\overline{Z}$ in an appropriate sense. 

\begin{lemma} \label{uniqueFuj}
The pushforward of $j_{Z*}c$ to $\overline{X} \times \overline{X}$ along
\[
\begin{tikzcd}
& \overline{Z} \arrow[dd,"q_1 \times q_2"] \arrow[rd,"q_2"] \arrow[ld,"q_1"] &
\\
\overline{X} \arrow[dd,equal] & & \overline{X} \arrow[dd,equal]
\\
& \overline{X}\times \overline{X} \arrow[rd,""] \arrow[ld,""] &
\\
\overline{X} & & \overline{X}
\end{tikzcd}
\]
is independent of the choice of $\overline{Z}$, where the pushforward exists since every arrow in the diagram is proper. 

In particular, the homomorphism on the cohomology groups 
\[
H^*(\overline{X},j_{X*}\mathcal{F}) \rightarrow H^*(\overline{X},j_{X*}\mathcal{G})
\]
induced from $j_{Z*}c$
is independent of the choice of $\overline{Z}$. 
\end{lemma}

\begin{proof}
This is a direct consequence of the functoriality of pushforward, i.e. $(q_1\times q_2)_*(j_{Z*}c) = ((q_1 \times q_2) \circ j_Z)_* c = ((j_X \circ p_1) \times (j_X \circ p_2))_* c$. 
\end{proof}

\subsection{Addition of cohomological correspondences} \label{fjdlseiq}
Let
\[
c_1: p_1^* \mathcal{F} \longrightarrow p_2^! \mathcal{G}
\]
be a cohomological correspondence supported on 
\[
\begin{tikzcd}
& Z_1 \arrow[rd,"p_2"] \arrow[ld,"p_1"] &
\\
X & & Y, 
\end{tikzcd}
\]
and 
\[
c_2: q_1^* \mathcal{F} \longrightarrow q_2^! \mathcal{G}
\]
a cohomological correspondence supported on 
\[
\begin{tikzcd}
& Z_2 \arrow[rd,"q_2"] \arrow[ld,"q_1"] &
\\
X & & Y,
\end{tikzcd}
\]
and we assume as usual that $p_2$ and $q_2$ are proper. 
One cannot naively define the addition of $c_1$ and $c_2$ since they are defined on different support. However, $c_1$ and $c_2$ induce two morphism from $H^{*}(X,\mathcal{F})$ to $H^*(Y, \mathcal{G})$, and clearly we can add them. We can upgrade this addition to a cohomological correspondence as follows. 

First note that we have a natural factorization
\[
\begin{tikzcd}
& & Z_1 \arrow[rrdd,"p_2"] \arrow[lldd,"p_1"] \arrow[d,"p_1 \times p_2"]
& &
\\
 & & X\times Y \arrow[rrd,"a_2"] \arrow[lld,"a_1"] & &
\\
X & & & & Y. 
\end{tikzcd}
\]
Since $a_2$ is not necessarily proper, it does not directly fall into the pushforward formalism. However, we observe that the image of $p_1 \times p_2$ is proper over $Y$ by properness of $p_2$, so $p_1 \times p_2$ factorizes through a closed subscheme $W$ of $X\times Y$ that is proper over $Y$, so we can pushforward $c_1$ to $W$. Moreover, we can pushforward cohomological correspondence from $W$ to $X\times Y$ by 1) of remark \ref{fjlsqw}. Thus we still have a pushforward of $c_1$ to $X\times Y$, which we again denote by $(p_1 \times p_2)_* c_1$. Similarly, we have  $(q_1 \times q_2)_*c_2$ on $X\times Y$, and we define
\[
c_1 + c_2 := (p_1 \times p_2)_* c_1+ (q_1 \times q_2)_*c_2
\]
as a cohomological correspondence supported on $X\times Y$. Clearly this induces the expected addition of morphism on cohomology groups. 

\begin{remark}
If $Z_1=Z_2$ and $p_i=q_i$, then there is a potential conflict of notation as we can also define $c_1+ c_2$ directly as addition of morphisms of sheaves on $Z_1$. This does not create real conflict as its pushforward to $X\times Y$ coincide with our definition, since pushforward commutes with addition of morphisms. 
\end{remark}

We now show that addition as we just defined commutes with pushforward. Suppose that we have two pushforward diagrams
\[
\begin{tikzcd}
& Z_1 \arrow[rd,"p_2"] \arrow[ld,"p_1"] \arrow[dd,"f_1"]
& 
& & & 
& Z_2 \arrow[rd,"q_2"] \arrow[ld,"q_1"] \arrow[dd,"f_2"]
& 
\\
X_1 \arrow[dd,"h"]& & Y_1 \arrow[dd,"g"] 
& & &
X_1 \arrow[dd,"h"]& & Y_1 \arrow[dd,"g"] 
\\
& W_1 \arrow[rd,"a_2"] \arrow[ld,"a_1"] & 
& & &
& W_2 \arrow[rd,"b_2"] \arrow[ld,"b_1"] & 
\\
X_2 & & Y_2
& & &
X_2 & & Y_2,
\end{tikzcd}
\]
with $p_2$, $a_2$, $q_2$ and $b_2$ proper, 
and let 
\[
c_1: p_1^* \mathcal{F} \longrightarrow p_2^! \mathcal{G}
\]
\[
(resp. \ \ 
c_2: q_1^* \mathcal{F} \longrightarrow q_2^! \mathcal{G})
\]
be a cohomological correspondence supported on $Z_1$ (resp. $Z_2$). 

\begin{lemma} \label{additionlemma}
We have canonical identification 
\[
f_{1*} c_1 + f_{2*} c_2 \cong (h \times g)_* (c_1 + c_2)
\]
\end{lemma}

\begin{proof}
An equivalent way to write the above two diagrams is the following two commutative diagrams 
\[
\begin{tikzcd}
Z_1 \arrow[r,"f_1"] \arrow[d,"p_1 \times p_2"]
& W_1 \arrow[d,"a_1\times a_2"] 
& & & 
Z_2 \arrow[r,"f_2"] \arrow[d,"q_1 \times q_2"]
& W_2 \arrow[d,"b_1\times b_2"] 
\\
X_1 \times Y_1 \arrow[r,"h\times g"] & X_2\times Y_2
& & &
X_1 \times Y_1 \arrow[r,"h\times g"] & X_2\times Y_2.
\end{tikzcd}
\]
Now we compute that
\[
f_{1*} c_1 + f_{2*} c_2 = (a_1\times a_2)_*f_{1*} c_1 + (b_1 \times b_2)_*  f_{2*}c_2= (h\times g)_*(p_1\times p_2)_* c_1 + (h\times g)_* (q_1 \times q_2)_* c_2
\]
\[
=(h\times g)_*((p_1\times p_2)_* c_1 +  (q_1 \times q_2)_* c_2)
= (h\times g)_* (c_1 + c_2).
\]
Note that 
$ (h\times g)_*(p_1\times p_2)_* c_1$ is defined since $(p_1\times p_2)_* c_1$ is supported on a closed subscheme $D$ of $X_1 \times Y_1$ that is proper over $Y_1$, and the image of  $D$ under  $h\times g$ is contained in a closed subscheme of $X_2 \times Y_2$ that is proper over $Y_2$. This follows from the properness of $p_2$ and $a_2$. Similarly for $(h\times g)_* (q_1 \times q_2)_* c_2$.  
\end{proof}

In particular, we see that the extension of cohomological correspondences in section \ref{928309isj} commutes with addition. 

\section{Morel's weight t-structures}

The weight t-structures were originally introduced by Morel in \cite{MR2350050} to study Galois representations coming from intersection cohomology of Shimura varieties, see \cite{MR2862060} and \cite{morel2010-lq}. A different motivic application of weight t-structures to noncompact Shimura varieties can be found in \cite{wu_2021}. 

Let $k=\mathbb{F}_q$ be a finite field, and $X$ be a separated scheme of finite type over $k$. Let $D_m^b(X, \overline{\mathbb{Q}}_l)$ be the category of mixed $l$-adic complexes in the sense of \cite{MR751966} 5.1.5. The weight t-structure on $D_m^b(X, \overline{\mathbb{Q}}_l)$ consists of full subcategories $(^{\omega}D^{\leq a}, ^{\omega}D^{\geq a+1})$  of $D_m^b(X, \overline{\mathbb{Q}}_l)$ for $a \in \mathbb{Z}$,
where $^{\omega}D^{\leq a}$ (resp. $^{\omega}D^{\geq a+1}$) 
is the category of complexes $K$ whose perverse cohomology groups $^pH^iK$ have weights $\leq a$ (resp. $\geq a+1$) for all $i$. In particular, we have functors $\omega_{\leq a} : D_m^b(X, \overline{\mathbb{Q}}_l) \longrightarrow ^{\omega}D^{\leq a}$, 
$\omega_{\geq a+1}: D_m^b(X, \overline{\mathbb{Q}}_l) \longrightarrow ^{\omega}D^{\geq a+1}$ such that for any $K \in D_m^b(X, \overline{\mathbb{Q}}_l)$ we have a distinguished triangle
\[
\omega_{\leq a} K \longrightarrow K \longrightarrow \omega_{\geq a+1}K \longrightarrow \cdot [1].
\]

A key feature of the weight t-structure is that we can characterize intersection sheaves easily. 

\begin{theorem}
(\cite{MR2350050} Theorem 3.1.4.) Let 
$j: U \hookrightarrow X$ be a nonempty open immersion of separated schemes of finite type over $k$, and $K$ is a perverse sheaf pure of weight $a$ on $U$, then the canonical maps
\[
\omega_{\geq a} j_!K \longrightarrow j_{!*}K 
\longrightarrow 
\omega_{\leq a} j_* K 
\]
are isomorphisms. 
\end{theorem}

In particular, when $U$ is smooth of dimension $d$, and $\mathbb{L}$ is a local system pure of weight $a$ on $U$, then 
\[
IC(\mathbb{L}) := j_{!*} \mathbb{L}[d] \cong \omega_{\leq a}j_*\mathbb{L}[d]. 
\]

Another feature of weight t-structures  important to us is that it behaves well with respect to cohomological correspondences.

\begin{lemma} \label{oeiuroeiui}
(\cite{MR2350050} lemma 5.1.3.)
Let 
\[
c: p_1^* \mathcal{F} \longrightarrow p_2^! \mathcal{G}
\]
be a cohomological correspondence supported on
\[
\begin{tikzcd}
& Z \arrow[rd,"p_2"] \arrow[ld,"p_1"] &
\\
X & & X.
\end{tikzcd}
\]
with $\mathcal{F} , \mathcal{G} \in D_m^b(X, \overline{\mathbb{Q}}_l)$. 
Let $a$ be an integer, then there exists a unique cohomological correspondence 
\[
\omega_{\leq a} c: p_1^* \omega_{\leq a}\mathcal{F} \longrightarrow p_2^!\omega_{\leq a}  \mathcal{G}
\]
that makes the diagram
\[
\begin{tikzcd}
p_1^* \omega_{\leq a}\mathcal{F} \arrow[d,"\omega_{\leq a} c"] \arrow[r] & p_1^* \mathcal{F} \arrow[d,"c"]
\\
p_2^!\omega_{\leq a}  \mathcal{G} \arrow[r] & p_2^!  \mathcal{G}
\end{tikzcd}
\]
commute, where the horizontal maps are the canonical natural transformation $\omega_{\leq a} \rightarrow \text{id}$. Similarly we have $\omega_{\geq a+1} c$. 
\end{lemma}

\begin{remark}
It is possible to make $\omega_{\leq a} c $ and $\omega_{\geq a+1} c$ explicit, see \cite{MR2350050} lemma 5.1.4. 
\end{remark}

We now check that $\omega_{\leq a} c$ commutes with composition and addition.
 
\begin{lemma} \label{weightlemma}
We have canonical identifications 
\[
\omega_{\leq a} (c_1 \circ c_2) = (\omega_{\leq a} c_1) \circ (\omega_{\leq a} c_2),
\]
\[
\omega_{\leq a}(c_1+c_2) = \omega_{\leq a} c_1 + \omega_{\leq a} c_2,
\]
and similarly for $\omega_{\geq a+1}$. 
\end{lemma}

\begin{proof}
We first check the commutation with composition. Let us write down the diagram of the composition $c_1 \circ c_2$,
\[
\begin{tikzcd}
& & Z_2 \times_{X} Z_1 \arrow[rd,"f"] \arrow[ld,"g"] & &
\\
& Z_2 \arrow[rd,"p_2"] \arrow[ld,"p_1"] & & Z_1 \arrow[rd,"q_2"] \arrow[ld,"q_1"] & 
\\
X & & X &  & X 
\end{tikzcd}
\]
and let 
\[
c_2: p_1^* \mathcal{F} \longrightarrow p_2^! \mathcal{G}
\]
\[
c_1: q_1^* \mathcal{G} \longrightarrow q_2^! \mathcal{H}.
\]
Then by definition 
\[
c_1 \circ c_2 : g^*p_1^* \mathcal{F} \overset{g^*c_2}{\longrightarrow} g^* p_2^!\mathcal{G} \cong f^!q_1^* \mathcal{G} \overset{f^!c_1}{\longrightarrow } f^!q_2^!\mathcal{H}
\]
which fits into a commutative diagram
\[
\begin{tikzcd}
g^*p_1^* \mathcal{F} \arrow[r,"g^*c_2"] 
&
g^* p_2^!\mathcal{G} \arrow[r, "\sim"]
&
f^!q_1^* \mathcal{G}  \arrow[r,"f^!c_1"]
&
f^!q_2^!\mathcal{H}
\\
g^*p_1^* \omega_{\leq a}\mathcal{F} \arrow[r,"g^*\omega_{\leq a}c_2"] 
\arrow[u]
&
g^* p_2^!\omega_{\leq a}\mathcal{G} \arrow[r, "\sim"]
\arrow[u]
&
f^!q_1^* \omega_{\leq a}\mathcal{G}  \arrow[r,"f^!\omega_{\leq a}c_1"]
\arrow[u]
&
f^!q_2^!\omega_{\leq a}\mathcal{H}
\arrow[u]
\end{tikzcd}
\]
where as usual the vertical maps are the natural transformation $\omega_{\leq a} \rightarrow \text{id}$. The first and third square commute  by lemma \ref{oeiuroeiui}, and the middle square commute by the naturality of the base change isomorphism $g^*p_2^! \cong f^!q_1^*$. Now by lemma \ref{oeiuroeiui} again, we see that the composition of lower horizontal arrows is $\omega_{\leq a } (c_{1} \circ c_2)$, but it is also $(\omega_{\leq a} c_1) \circ (\omega_{\leq a} c_2)$ 
by the definition of composition. 

The preservation of addition is similar, and we only sketch the argument. Suppose that we are in the situation of section \ref{fjdlseiq}.  With the same notation as in $loc.cit.$,   we have 
\[
\omega_{\leq a} c_1 + \omega_{\leq a } c_2 = (p_1 \times p_2)_* \omega_{\leq a}c_1 + (q_1 \times q_2)_*\omega_{\leq a} c_2.
\]
One can check using lemma \ref{oeiuroeiui} as in the previous paragraph that  
$(p_1 \times p_2)_* \omega_{\leq a}c_1 =  \omega_{\leq a}(p_1 \times p_2)_*c_1$, and 
$(q_1 \times q_2)_*\omega_{\leq a} c_2 = \omega_{\leq a}(q_1 \times q_2)_* c_2$, then 
\[
\omega_{\leq a} c_1 + \omega_{\leq a } c_2 = \omega_{\leq a}((p_1 \times p_2)_* c_1 + (q_1 \times q_2)_* c_2) 
= \omega_{\leq a} (c_1 +c_2). 
\]
\end{proof}

\section{Proof of theorem \ref{eouril}} 

Let $a$ be the weight of the pure local system $\mathbb{L}$, and $d$ the dimension of $X$. We make essential use of the characterization 
\[
IC(\mathbb{L}) \cong \omega_{\leq a} j_* \mathbb{L}[d],
\]
where $j: X \hookrightarrow \overline{X}$ is the open embedding. 

First, we know from lemma \ref{Fujiwa} and \ref{uniqueFuj} that $\Psi$ and the coefficients of $P(T)$ extends canonically to cohomological correspondences from $j_{*} \mathbb{L}[d]$ to itself if we choose compactifications of the support of correspondences, and it is independent of the choice  once we pushforward them to $\overline{X}\times \overline{X}$. 

Next, lemma \ref{oeiuroeiui}  tells us that the extended cohomological correspondences induce naturally cohomological correspondences from $\omega_{\leq a} j_* \mathbb{L}[d]$ to itself. Thus the action of the coefficients of $P(T)$ and $\Psi$ on $H^*(\overline{X}, IC(\mathbb{L}))$ is canonically defined. Indeed, they already induce canonical cohomological correspondences from $\omega_{\leq a} j_* \mathbb{L}[d]$ to itself with support on $\overline{X} \times \overline{X}$, which we  denote by $\omega_{\leq a} j_* \Psi$ and (the coefficients of) $\omega_{\leq a} j_* P(T)$. Moreover, for any cohomological correspondence $c$ from $\mathbb{L}$ to itself with support proper over $X$ under the second leg,  we abuse notation by denoting $\omega_{\leq a} j_* c$ 
the canonical cohomological correspondence as constructed above from $\omega_{\leq a} j_* \mathbb{L}[d]$
to itself supported on $\overline{X} \times \overline{X}$. 

Now, we see from lemma \ref{additionlemma} and \ref{weightlemma} that the operation $\omega_{\leq a} j_* $ preserves addition and composition of cohomological correspondences, therefore 
\[
\omega_{\leq a} j_* (P(\Psi)) = (\omega_{\leq a} j_*P) (\omega_{\leq a} j_* \Psi) 
\]
as cohomological correspondences from 
$\omega_{\leq a} j_* \mathbb{L}[d]$ to itself supported on $\overline{X} \times \overline{X}$. By our assumption $P(\Psi)=0$, so 
\[
(\omega_{\leq a} j_*P) (\omega_{\leq a} j_* \Psi) =0 
\]
as desired, which induces the corresponding relation on the cohomology groups. 

To deduce corollary \ref{ES}, we need to check that $\omega_{\leq a} j_*$ of Hecke operators and Frobenius  on open Shimura varieties are the same as its natural extensions to minimal compactifications, but this follows directly from the unique extension property of 
lemma \ref{oeiuroeiui}  and \ref{Fujiwa}.

\begin{remark}
The formation of addition in our treatment is slightly ad hoc.  A better framework is introduced in \cite{Zhu}, in which we consider a groupoid $C \rightrightarrows X$ of spaces satisfying suitable properness addition, and cohomological correspondences from $\mathcal{F} $ to $\mathcal{G}$ with support on $C$. The groupoid structure allows us to pushforward the composition back to cohomological correspondences supported on $C$, and we can add cohomological correspondences directly. This is the situation when we work with Shimura varieties, where we take $C$ to be the stack parameterizing $p$-power isogenies of universal abelian varieties over Shimura varieties. The Hecke operators and Frobenius all live as cohomological correspondences supported on $C$. 

For our purpose, the problem with this treatment is that we need to compactify  $C$ to a groupoid over the fixed compactification $\overline{X}$ of $X$, in general this seems hopeless. In the special case of Shimura varieties with $C$ being stack of $p$-power isogenies, it is possible that we can make use of the finiteness of the correspondence to obtain the desired compaction (maybe weaker than being a groupoid) over $\overline{X}$. Our approach is more flexible, and avoids the issue of choosing nice compactifications.   
\end{remark}

\bibliographystyle{alpha} 
\bibliography{ref}
\end{document}